\documentclass[a4paper, 12pt]{amsart}
\usepackage{amsfonts, amsmath, amssymb, amsthm}  
\usepackage[english]{babel}
\usepackage[utf8]{inputenc}
\usepackage[OT1]{fontenc}
\usepackage{bm}

\usepackage[margin=2.2cm]{geometry}

\usepackage{textcomp, cmap}	
\usepackage{graphicx, wrapfig}
\usepackage{xcolor, hyperref}
\usepackage{array,tabularx}
\usepackage{euscript, mathrsfs}

\newcounter{iii}

\newcommand{\bb}{{\mathcal B}}
\newcommand{\aaa}{{\mathcal A}}
\newcommand{\mm}{\mathcal M}
\newcommand{\G}{\mathcal G}

\newcommand{\ff}{\mathcal F}
\theoremstyle{plain}
\newtheorem{thm}{Theorem}

\newtheorem{lem}[thm]{Lemma}

\newtheorem*{obs}{Observation}
\newtheorem{pro}{Problem}
\newtheorem{cor}[thm]{Corollary}
\theoremstyle{definition}

\numberwithin{equation}{section}
\numberwithin{thm}{section}

\title{Intersecting families with covering number $3$}
\author{Andrey Kupavskii}
\address{Moscow Institute of Physics and Technology, Saint-Petersburg State University, Innopolis University, Russia; Email: {\tt kupavskii@ya.ru}.}
\date{}

\begin{document}
\maketitle
\begin{abstract} The covering number of a family is the size of the smallest set that intersects all sets from the family. In 1978 Frankl determined for $n\ge n_0(k)$ the largest intersecting family of $k$-element subsets of $[n]$  with covering number $3$. In this paper, we essentially settle this problem, showing that the same family is extremal for any $k\ge 100$ and $n>2k$.
\end{abstract}

\section{Introduction}

For integers $a\le b$,  put $[a,b]:=\{a,a+1,\ldots, b\}$, and denote $[n]:=[1,n]$ for shorthand. For a set $X$, denote by $2^{X}$ its power set and, for integer $k\ge 0$,  denote by ${X\choose k} ( {X\choose \le k}$) the collection of all (at most) $k$-element subsets ({\it $k$-sets}) of $X$. A {\it family} is simply a collection of sets.  We call a family {\it intersecting}, if any two of its sets intersect. A `trivial' example of an intersecting family is the {\it full star}: the family of all sets containing a fixed element. The {\it covering number} $\tau(\ff)$ of a family $\ff$ is the size of the smallest $X$ such that $X\cap F\ne \emptyset$ for all $F\in \ff$. Each such $X$ we call a {\it cover} or a {\it hitting set}. An intersecting family is {\it trivial} or a {\it star} if it has covering number $1$. We call an intersecting family {\it non-trivial}, if the intersection of all sets from the family is empty. Alternatively, a family is non-trivial if its covering number is at least $2$.

One of the oldest and most famous results in extremal combinatorics is the Erd\H os--Ko--Rado theorem \cite{EKR}, which states that for  $n\ge 2k>0$ an intersecting family $\ff\subset {[n]\choose k}$ satisfies $|\ff|\le {n-1\choose k-1}$. Thus, the extremal example is a full star. Answering a question of Erd\H os, Ko, and Rado, Hilton and Milner \cite{HM} found the size and structure of the largest non-trivial intersecting families of $k$-sets. It has size 
 ${n-1\choose k-1}-{n-k-1\choose k-1}+1$ and has an element that intersects all but one set of the family.

For a family $\ff\subset 2^{[n]}$ and $i\in [n]$, the {\it degree} of $i$ in $\ff$ is the number of sets from $\ff$ containing $i$. Let $\Delta(\ff)$ stand for the {\it maximal degree} of an element in $\ff$. The {\it diversity} $\gamma(\ff)$ of $\ff$ is the quantity $|\ff|-\Delta(\ff)$. One may think of diversity as of the distance from $\ff$ to the closest star.  Frankl \cite{Fra1} proved a far-reaching generalization of the Hilton--Milner theorem. We present its slightly stronger variant due to Zakharov and the author \cite{KZ}.
\begin{thm}[\cite{KZ}]\label{thm1} Let $n>2k>0$ and $\ff\subset {[n]\choose k}$ be an intersecting family. If $\gamma(\ff)\ge {n-u-1\choose n-k-1}$ for some real $3\le u\le k$, then \begin{equation}\label{eq01}|\ff|\le {n-1\choose k-1}+{n-u-1\choose n-k-1}-{n-u-1\choose k-1}.\end{equation}
\end{thm}
We note that the Hilton--Milner theorem is included in Theorem~\ref{thm1}: simply put $u=k$. Theorem~\ref{thm1} provides essentially the strongest possible stability result for the Erd\H os--Ko--Rado theorem in the regime when the intersecting family is large (more precisely, for the families of size at least ${n-2\choose k-2}+2{n-3\choose k-2}$). There are several other stability results for the Erd\H os--Ko--Rado theorem, see, e.g. \cite{DT,EKL,Fri}. Recently, there were a series of results \cite{HK,KostM} in which the authors found a more fine-grained relationship between the structure of the intersecting family and its size (again, for large intersecting families). See the papers of the author \cite{Kup73, Kup74} for the most general and conclusive results. 

Theorem~\ref{thm1} is stated in terms of diversity. The largest bound on diversity it gives is for $u=3$: $\gamma(\ff)\ge{n-4\choose k-3}$. Actually, for relatively large $n$, the diversity cannot be much bigger: I have showed \cite{Kup21} that $\gamma(\ff)\le {n-3\choose k-2}$ for $n>Ck$ with some large $C$, and then for $n>72k$ by Frankl \cite{Fdiv} and $n>36k$ by Frankl and Wang \cite{FWan}.

Another natural measure of how far the intersecting family is from the trivial family is the covering number. Intersecting families of $k$-sets with fixed covering number were studied in several classical works. The Erd\H os--Ko--Rado theorem shows that the largest intersecting family of $k$-element sets has covering number $1$. The  result of Hilton and Milner \cite{HM} determined the largest intersecting family with covering number $2$. It is clear that any $k$-uniform intersecting family $\ff$ satisfies $\tau(\ff)\le k$: indeed, any set of $\ff$ is a hitting set for $\ff$. In a seminal paper \cite{EL}, Erd\H os and Lov\'asz proved that an intersecting family $\ff\subset {[n]\choose k}$ with $\tau(\ff)=k$ has size at most $k^k$ (note that it is independent of $n$!) and provided a lower bound of size roughly $(k/e)^k$. Later, both lower \cite{FOT2} and upper \cite{Che, AR, Fra21, Z} bounds were improved.

In \cite{F16} (cf. also \cite{FT7}), Frankl studied the following general question: what is the size $c(n,k,t)$ of the largest intersecting family $\ff\subset {[n]\choose k}$ with $\tau(\ff)\ge t$?
Let us define the following important family.
\begin{equation}\label{deft2}
\mathcal T_2(k):=\big\{[k]\big\}\cup \big\{\{1\}\cup [k+1,2k-1]\big\}\cup \big\{\{2\}\cup [k+1,2k-1]\big\}.\end{equation}
It is easy to see that $\mathcal T_2(k)$ is intersecting, moreover, $\tau(\mathcal T_2(k))=2$.
Define $\mathcal C_3(n,k)\subset {[n]\choose k}$ to be the maximal intersecting family in which the subfamily of sets not containing $1$ is isomorphic to $\mathcal T_2(k)$. It is easy to see that $\tau(\mathcal C_3(n,k))=3$. Frankl proved the following theorem.

\begin{thm}[\cite{F16}]\label{thmtau33} Let $k\ge 3$ and $n\ge n_0(k)$. Then $c(n,k,3)= |\mathcal C_3(n,k)|$. Moreover, for $k\ge 4$ the equality holds only for families isomorphic to $\mathcal C_3(n,k)$.
\end{thm}

Frankl proved this theorem using the Delta-system method, which was behind many of the breakthroughs in extremal set theory in the 1970s and 80s. If one analyzes the method used by Frankl directly, then the bound on $n$ is doubly exponential in $k$ if one analyzes the proof in \cite{F16}. Using a refined variant of the $\Delta$-system method, one can get the bound down to $n$ polynomial in $k$, but the polynomial will most likely be at least cubic. The main result of this paper is an extension of the theorem above into an essentially full range of parameters.

\begin{thm}\label{thmtau3} The conclusion of Theorem~\ref{thmtau33} holds for any $k\ge 100$, $n>2k$.
\end{thm}

The case $n=2k$ is special. If $n=2k$ then all maximal intersecting families have size $\frac 12{2k\choose k}$ and are obtained by taking exactly one set out of each pair of complementary sets. If, say, we select these sets randomly then for a somewhat large $k$ with probability close to $1$ we will obtain a family with covering number at least $3$ (or bigger than, say, $k/10$).

In the paper \cite{FOT1}, the authors managed to extend the result of \cite{F16} to the case $\tau =4$, determining the exact value of $c(n,k,4)$ and the structure of the extremal family for $n>n_0(k)$. The analysis in \cite{FOT1} is much more complicated than that in \cite{F16}, and the problem for $\tau\ge 5$ is still wide open. It is possible that the result of \cite{FOT1} may be extended to much smaller $n$ using the techniques from this paper. In the case of $\tau\ge 5$, however, we do not even know the answer for $n>n_0(k)$, and this appears to us the most challenging problem in this direction. The main difficulty for $t\ge 5$ lies in the following problem.

\begin{pro}\label{propbol} Given an intersecting family $\ff$ of $k$-sets with $\tau(\ff)=t$, what is the maximum number of hitting sets of size $t$ it may have?
\end{pro}

\vskip+0.1cm
{\bf Remark: } A weaker version of Theorem~\ref{thmtau3} (the same result for $n>Ck$ with unspecified $C$ coming from the junta method), along with many of the proof ideas, appeared in an unpublished manuscript of the author \cite{Kup73}. Actually, in \cite{Kup73} I posed a problem to determine $c(n,k,3)$ for all $n>2k\ge 8$. Thus, Theorem~\ref{thmtau3} essentially answers this question. 
 I decided to split the manuscript and publish it separately because of its heterogeneous nature and, importantly, because its first part, dedicated to the complete version of Frankl's degree theorem, turned out to have been obtained earlier in an unjustly overlooked paper by Goldwasser \cite{Gold}.\vskip+0.1cm

{\bf Remark 2: } After I submitted the paper to the journal, it was pointed to me by one of the referees that recently Frankl and Wang obtained essentially the same result in \cite{FW23}. (Their paper was uploaded to arXiv in September 2023, previous versions of that paper contained weaker results.) The approaches are very different, and the present proof is significantly shorter and has less case analysis/calculations. The approach of Frankl and Wang relies on shifting, an intricate analysis of shift-resistant pairs (i.e., pairs of elements over which we cannot do shifts without losing the $\tau\ge 3$ property), and a lot of calculations that exploit certain partial structural information on the family. In this paper, I develop the bipartite switching technique and combine it with ideas coming from spread approximations. This approach seems to be quite flexible and potentially useful in many other extremal set theory problems.


\section{Proof of Theorem~\ref{thmtau3}}
\subsection{Outline of the approach}
Take an intersecting family $\ff\subset {[n]\choose k}$ with covering number $3$. The proof of the theorem bifurcates based on whether $\gamma(\ff)$ is large or small. (Concretely, whether $\gamma(\ff)>{n-5\choose k-3}$ or not.)

In the case of large diversity, we show that $\ff$ is significantly smaller than the family $\mathcal C_3(n,k)$ that is expected to be extremal. For $n<2k^2$ the family $\mathcal C_3(n,k)$ is still quite large, and it is sufficient to use the bound \eqref{eq01} to show that $|\ff|$ is small.

For $n>2k^2$, we use the recently developed `peeling' procedure, which is an independent and apparently useful part of the spread approximation technique \cite{KuZa}, which was recently upgraded to be much more efficient \cite{Kup77}. We say that a (non-uniform) family is maximal intersecting if no set can be replaced by a proper subset without violating the intersection property. Peeling is an iterative procedure in which we replace the family by a maximal intersecting family and then remove (peel) the layer of the largest sets. It allows for an efficient control of the structure of the family. In a way, it is an improved analogue of the Delta-system method as used by, say, Frankl in \cite{F16}. Back to the proof, the fact that $\ff$ has large diversity implies that we can peel the family up to a constant layer without the remaining family getting trivial (i.e., being replaced by one singleton). This at the end allows us to efficiently control the number of different $3$-element covers of $\ff$, which turns out to be much smaller than that for $\mathcal C_3(n,k)$.  As a result, the size of $\ff$ is much smaller than that of $\mathcal C_3(n,k)$.

The second case is small diversity. The key idea in this case is an extension of the bipartite switching idea, which was introduced in \cite{KZ} (similar ideas appeared earlier in \cite{FK1}). Its application is based on Corollary~\ref{corkz} for cross-intersecting families (essentially, the two cross-intersecting families are restrictions of $\ff(1)$ and $\ff(\bar 1)$, where $1$ is the element with the highest degree).  We carefully change the family $\ff$ so that its size does not decrease and the covering number is preserved, but $\ff(\bar 1)$ gets smaller and smaller. We actually start with a minimal subfamily $\mathcal M\subset \ff(\bar 1)$ with covering number $2$ and at the end of the procedure we get that $\ff(\bar 1)$ coincides with $\mathcal M$.

The last step is to show that the optimal minimal $\mathcal M$ for $\ff(\bar 1)$ is isomorphic to $\mathcal T_2(k)$. This is done in  Lemma~\ref{lemmin}. In order to show this, we found an elegant way to upper bound the size of $\ff$ doing a specific size count of the size of the family $\ff(1)$ that cross-intersects $\mathcal M$ and compare it term by term with an analogous count for $\mathcal T_2(k)$.

\subsection{Preliminaries}
For a family $\aaa\subset 2^{X}$ and a set $Y\subset X$, we use the following  notation:
\begin{align*}
  \aaa(Y) =&\big\{F\setminus Y: Y\subset F, F\in \aaa\big\}, \\
  \aaa[Y] =&\big\{F: Y\subset F, F\in \aaa\big\}, \\
  \aaa(\bar Y) =& \big\{F: Y\cap F=\emptyset, F\in \aaa\big\}.
\end{align*}
Note the difference between $\aaa(Y)$ and $\aaa[Y]$: we exclude $Y$ from the sets in the former and include it in the latter. For an element $x$ we write $\aaa(x)$, $\aaa(\bar x)$ instead of $\aaa(\{x\}),\aaa(\bar{\{x\}})$ for shorthand. For two families $\aaa,\G$ we also use the following notation:
$$\aaa[\G] = \cup_{F\in \G}\aaa[F].$$
The families $\aaa,\bb$ are {\it cross-intersecting} if for any $A\in\aaa,B\in\bb$ we have $A\cap B\ne \emptyset$. The following result was essentially (modulo uniqueness) obtained in \cite{KZ}, and is a consequence of a more general and stronger result from \cite{Kup73}.
\begin{cor}[\cite{KZ}]\label{corkz}
  Let $a,b>0$, $n>a+b$. Let $\aaa\subset {[n]\choose a},\ \bb\subset {[n]\choose b}$ be a pair of cross-intersecting families. Denote $t:=b+1-a.$ Then, if $|\bb|\le {n-t\choose a-1}$, then
\begin{equation}\label{eqcreasy} |\aaa|+|\bb|\le {n\choose a}.\end{equation}
Moreover, the displayed inequality is strict unless $|\bb|=0$.

If ${n-j\choose b-j}\le |\bb|\le {n-t\choose a-1}$ for integer $j\in [t, b]$, then
\begin{equation}\label{eqcreasy2} |\aaa|+|\bb|\le {n\choose a}-{n-j\choose a}+{n-j\choose b-j}.\end{equation}
Moreover, if the left inequality on $\bb$ is strict, then the inequality in the displayed formula above is also strict, unless $j=t+1$ and $|\bb| = {n-t\choose a-1}$.
\end{cor}
The result in \cite{KZ} did not explicitly treat the equality case. However, it is clear that strictness of \eqref{eqcreasy} follows from \eqref{eqcreasy2}, and the equality case in \eqref{eqcreasy2} follows from \cite[Theorem 2.12, part 3]{Kup73}.

Let us  recall the definition of shifting. For a given pair of indices $1\le i<j\le n$ and a set $A \subset [n]$, define its {\it$(i,j)$-shift} $S_{ij}(A)$ as follows. If $i\in A$ or $j\notin A$, then $S_{ij}(A) = A$. If $j\in A, i\notin A$, then $S_{ij}(A) := (A-\{j\})\cup \{i\}$. That is, $S_{ij}(A)$ is obtained from $A$  by replacing $j$ with $i$.
The  $(i,j)$-shift $S_{ij}(\mathcal A)$ of a family $\mathcal A$ is as follows:
$$S_{ij}(\mathcal A) := \{S_{ij}(A)\ :\  A\in \mathcal A\}\cup \{A\ :\  A,S_{ij}(A)\in \mathcal A\}.$$
Shifting is a very useful operation in the study of a class of extremal set theoretic problems. We refer to \cite{F3} for a survey. Shifting preserves the sizes of sets and the size of the family: $|S_{ij}(\aaa)| = |\aaa|$. Importantly, if $\aaa$ is intersecting then $S_{ij}(\aaa)$ is intersecting. Unfortunately, shifting can potentially reduce the covering number and is thus not directly applicable to our problem. However, with a certain amount of care, it is still possible to use, which we do in this paper.

The following lemma is one of the key ingredients in the proof of the theorem. It allows us to compare different intersecting families with `minimal' $\ff(\bar 1)$. Let us first give some definitions.
Given integers $m>2s$, let us denote by $\mathcal T_2'(s):=\{[s], [s+1,2s]\}$. Let $\ff_2'(s)\subset {[m]\choose k-1}$ stand for the largest family that is cross-intersecting with $\mathcal T_2'(s)$. Let $\ff_2(s)\subset {[m]\choose k-1}$ stand for the largest family that is cross-intersecting with $\mathcal T_2(s)$  (cf. \eqref{deft2}).

\begin{lem}\label{lemmin} Let $k\ge s$ and $m\ge k+s$ be integers, $k\ge 4$. Given a family $\mathcal H\subset {[m]\choose s}$ with $\tau(\mathcal H)=2$ and minimal w.r.t. this property, consider  the maximal family $\ff\subset {[m]\choose k-1}$ that is cross-intersecting with $\mathcal H$. Then the unique maximum of $|\ff|+|\mathcal H|$ is attained when $\mathcal H$ is isomorphic to $\mathcal T_2'(s)$ (and $\ff$ is thus isomorphic to $\ff_2'(s)$).

If we additionally require that $\mathcal H$ is intersecting\footnote{Note that this is equivalent to requiring that $|\mathcal H|>2$.} then the maximum of $|\ff|+|\mathcal H|$ is attained for $\mathcal H$ and $\ff$ isomorphic to $\mathcal T_2(s)$ and $\ff_2(s)$. The maximal configuration is unique if $s\ge k$.
\end{lem}



\begin{proof}[Proof of Lemma~\ref{lemmin}]
Let us first express $|\ff'_2(s)|$. It is not difficult to see that
\begin{footnotesize}\begin{align}
\notag  |\ff'_2(s)|\  =\ & {m-1\choose k-2}-{m-s-1\choose k-2} + \\
\notag   & {m-2\choose k-2}-{m-s-2\choose k-2}+\\
\notag   &\cdots \\
\label{eqfs}   &{m-s\choose k-2}-{m-2s\choose k-2}.
\end{align}\end{footnotesize}
Indeed, in the first line we count the sets containing $1$ that intersect $[s+1,2s]$, in the second line we count the sets not containing $1$, containing $2$ and intersecting $[s+1,2s]$ etc.

Quite surprisingly, we can bound the size of $\ff$ for any $\mathcal H$ in a similar way. Suppose that $z:=|\mathcal H|$ and $\mathcal H = \{H_1,\ldots, H_z\}$. Since $\mathcal H$ is minimal, for each $l\in[z]$ there exists an element $i_l$ such that $i_l\notin H_l$ and $i_l\in \bigcap_{j\in[z]\setminus \{l\}} H_l$. (All $i_l$ are of course different.) Applying Bollobas' set-pairs inequality \cite{Bol} to $\mathcal H$ and $\{i_{l}\ :\ l\in[z]\}$, we get that $|\mathcal H|\le {s+1\choose s}=s+1$.

For each $l=2,\ldots, z$, we count the sets $F\in\ff$ such that $F\cap \{i_2,\ldots, i_l\}=\{i_l\}$. Such sets must additionally intersect $H_{l}\setminus \{i_2,\ldots, i_{l-1}\}$. Note that $H_1\supset \{i_2,\ldots, i_z\}$. This covers all sets from $\ff$ that intersect $\{i_2,\ldots, i_z\}$ and gives the first $z-1$ lines in the displayed inequality below. Next, we have to deal with sets from $\ff$ that do not intersect $\{i_2,\ldots, i_z\}$. Firstly, they must intersect $H_1$.
Assuming that $H_1\setminus \{i_2,\ldots,i_z\} = \{j_1,\ldots, j_{s+1-z}\}$, for each  $l=1,\ldots, s+1-z$ we further count the sets $F\in \ff$ such that $F\cap \{i_2,\ldots, i_z, j_1,\ldots, j_l\}=\{j_l\}$. Such sets must additionally intersect $H_{i}\setminus \{i_2,\ldots, i_{z}\}$ for some $i\in [2,z]$. (The element $j_l$ cannot be contained in all sets from $\mathcal H$ since the intersection of $\mathcal H$ is empty.)  Note that $H_{i}\setminus \{i_2,\ldots, i_{z}\}$ is a set of size $s-z+2$. This explains the last $s+1-z$ lines in the displayed inequality below.
 Since $F\cap H_1\ne \emptyset$ for any $F\in \ff$ and given that the classes for different $l$ are disjoint, we clearly counted each set from $\ff$ exactly once. (However, we may also count some sets that are not in $\ff$.) Doing this count, we get the following bound on $\ff$.
\begin{footnotesize}\begin{align}
 \notag |\ff|\ \le \ & {m-1\choose k-2}-{m-s-1\choose k-2} + \\
\notag   & {m-2\choose k-2}-{m-s-1\choose k-2}+\\
\notag   &\cdots\\
\notag   & {m-z+1\choose k-2}-{m-s-1\choose k-2}+\\
\notag   & {m-z\choose k-2}-{m-s-2\choose k-2}+\\
\notag   &\cdots \\
\label{eqfz3}   &{m-s\choose k-2}-{m-2s-2+z\choose k-2}\ =:\ f(z).
\end{align}\end{footnotesize}
Remark that \eqref{eqfz3} coincides with \eqref{eqfs} when substituting $z=2$.
We have $f(z-1)-f(z)\ge {m-s-1\choose k-2}-{m-s-2\choose k-2}={m-s-2\choose k-3}> 1$ (here we use that $m\ge s+k$ and $k\ge 4$). Therefore, for any $z\ge z'$, \begin{equation}\label{eqz'}|\mathcal H|+|\ff|\le f(z')+z',\end{equation} and the inequality is strict unless $z=|\mathcal H|=z'$.

At the same time, we have $|\ff'_2(s)|+|\mathcal T'_2(s)|=f(2)+2$ and $|\ff_2(s)|+|\mathcal T_2(s)|=f(3)+3$! (The former we have seen above, and the latter is easy to verify by doing exactly the same count.) Since, up to isomorphism, there is only one family $\mathcal H\subset {[m]\choose s}$ of size $2$ with $\tau(\mathcal H)=2$, we immediately conclude that the first part of the statement holds. To deduce the second part, we only need to show that, among all possible choices of $\mathcal H$ of size $3$, the only one (up to isomorphism) that attains equality in \eqref{eqz'} is $\mathcal H = \mathcal T_2(s)$.

Recall that, for uniqueness in the second part of the lemma, we have the additional condition $s\ge k$.
If there are two sets $H',H''\in \mathcal H$ such that $|H'\cap H''|=s-1$, then $\mathcal H$ is isomorphic to $\mathcal T_2(s)$. Therefore, in what follows we assume that $|H'\cap H''|\le s-2$ for any $H',H''\in\mathcal H$.

Let us deal with the case when $|H_i\cap H_j|=1$ for $1\le i<j\le 3$ and $H_1\cap H_2\cap H_3 = \emptyset$. Note that this implies that
$$  m\ge 3s-3.$$
 Since $s\ge k \ge 4$, there are elements $j_l\in H_l\setminus (H_{l'}\cup H_{l''})$, $\{l,l',l''\}=[3]$. Perform the $(j_1,j_2)$-shift on $\ff\cup \mathcal H$ and denote $\ff':=S_{j_1j_2}(\ff)$. Clearly, the sizes of the families stay the same and the resulting families are cross-intersecting. The family $S_{j_1j_2}(\mathcal H)$  has covering number $2$, moreover there are two sets in $\mathcal H$ that intersect in $2$ elements.

Finally, we may assume that $|H_1\cap H_2|\in [2,s-2]$. Then we do a similar count as in \eqref{eqfz3}. Recall that $z=3$. The first two steps (with $i_2,i_3$) are the same. The part with $j_i$ is, however, slightly modified. Take $j'\in (H_1\cap H_2)\setminus \{i_3\}$ and $j'' \in H_1\setminus (H_2\cup \{i_2\})$. Such choices are possible due to $|H_1\cap H_2|\in [2,s-2]$. Count the sets $F\in \ff$ such that $F\cap \{i_2,i_3,j'\}=j'$. They must intersect $H_3\setminus \{i_2\}$. Next, crucially, count the sets in $F\in \ff$ such that $F\cap \{i_2,i_3,j',j''\}=j''$. They must intersect $H_2\setminus \{i_3,j'\}$ (note the size of this set is $s-2$ instead of $s-1$). The remaining count is the same: let $\{j_1,\ldots, j_{s-4}\}:=H_1\setminus \{i_2,i_3,j',j''\}$ and, for each $l\in[s-4]$, count the sets $F\in \ff$ such that $F\cap \{i_2,i_3,j',j'',j_1,\ldots, j_{l}\} =j_l$. They must additionally intersect either $H_2\setminus \{i_3,j'\}$, or $H_3\setminus \{i_2\}$. Thus, we obtain the following bound.
\begin{footnotesize}\begin{align}
 \notag |\ff|\ \le \ & {m-1\choose k-2}-{m-s-1\choose k-2} + \\
\notag   & {m-2\choose k-2}-{m-s-1\choose k-2}+\\
\notag   & {m-3\choose k-2}-{m-s-2\choose k-2}+\\
\notag   & {m-4\choose k-2}-{m-s-\textbf{2}\choose k-2}+\\
\notag   & {m-5\choose k-2}-{m-s-4\choose k-2}+\\
\notag   &\cdots\\
\label{eqfz}   &{m-s\choose k-2}-{m-2s+1\choose k-2}\ =:\ f'(3).
\end{align}\end{footnotesize}
We have $f(3)-f'(3) = {m-s-2\choose k-2}-{m-s-3\choose k-2} = {m-s-3\choose k-3}\ge 1$ due to $m\ge s+k$, and thus $|\ff|\le f'(3)<f(3) = |\ff_2(s)|$. Thus, in the assumption $s\ge k$ and if $\mathcal H$, $|\mathcal H|\ge 3$, is not isomorphic to $\mathcal T_2(s)$, we  have strict inequality in \eqref{eqz'} for $z'=3$. The lemma is proven.
\end{proof}

\subsection{Proof of Theorem~\ref{thmtau3}}\label{sec51}
Recall the expression of the size of $\mathcal C_3(n,k)$, obtained in the proof of Lemma~\ref{lemmin} (cf. \eqref{eqfz}):
\begin{footnotesize}
\begin{align}
   \notag |\mathcal C_3(n,k)|\ =\ 3\ +\ &{n-2\choose k-2}-{n-k-2\choose k-2} + \\
 \notag  & {n-3\choose k-2}-{n-k-2\choose k-2}+\\
\notag   & {n-4\choose k-2}-{n-k-3\choose k-2}+\\
 \notag  &\cdots \\
  \label{eqsizec} &{n-k-1\choose k-2}-{n-2k\choose k-2}.
\end{align}
\end{footnotesize}
We can verify this formula directly. W.l.o.g. assume that the three sets not containing $1$ are $A_1=[2,k+1], A_2=\{2\}\cup[k+2,k], A_3=\{3\}\cup [k+2,2k]$. Then the first line counts the three sets $A_i$ and the sets containing 1, containing 2 and intersecting $A_3$ (i.e., all sets containing $1,2$ minus the sets that contain $1,2$ and avoid $A_3$). The second line counts the sets containing $1$ and $3$, avoiding $2$ and intersecting $A_2$, which effectively means intersecting $[k+2,2k]$. For each $i=2,\ldots, k$, in the line number $i$ we count the sets $B$ such that $B\cap [i+1] = \{1,i+1\}$ and that intersect $[k+2,2k]$.

The proof is very different in the case when the diversity of $\ff$ is large and when it is small. We first consider the case of large diversity.
\subsection{The case $\gamma(\ff)>{n-5\choose k-3}$}
First, we consider the case $n\le 2(k-1)^2$.
\begin{lem} If $k\ge 100$, $n\le 2(k-1)^2$ and $\ff\subset {[n]\choose k}$ is intersecting with $\gamma(\ff)>{n-5\choose k-3}$, then $|\ff|<|\mathcal C_3(n,k)|$.
\end{lem}\begin{proof}
The proof of the lemma is rather technical and requires different estimates on sums of binomial coefficients. We start by lower bounding the size of $\mathcal C_3(n,k)$. For any $i\ge 0$
$$\frac{{n-k-2-i\choose k-2}}{{n-3-i\choose k-2}}=\prod_{j=1}^{k-2}\frac{n-k-1-i-j}{n-2-i-j}\le e^{-(k-1)(k-2)/(n-2)}\le e^{-\frac{k-2}{2(k-1)}}<\frac 23,$$
provided $k\ge 100$. Thus, we can lower bound each line in \eqref{eqsizec} of the form $a-b$ by $\frac 13 a$ ($a,b$ are some binomial coefficients here) and get
\begin{align*} |\mathcal C_3(n,k)|\ >&\ \frac13\Big({n-2\choose k-2}+{n-3\choose k-2}+\ldots+{n-k-1\choose k-2}\Big)\\
=&\ \frac13\Big({n-1\choose k-1}-{n-k-1\choose k-1}\Big)\\
\ge&\ \frac 19 {n-1\choose k-1},
\end{align*}
where the last inequality is obtained analogously. On the other hand, since $\ff$ is intersecting and $\gamma(\ff)\ge {n-5\choose k-3}\ge {n-5\choose k-4},$ we can use Theorem~\ref{thm1} with $u=4$ and get
\begin{align}
\label{eqboundf}  |\ff|\ \le&\ {n-1\choose k-1}-{n-5\choose k-1}+{n-5\choose k-4} \\
  \notag\le &\ 5{n-2\choose k-2}\ \le \ \frac 1{10}{n-1\choose k-1}
\end{align}
for $n>50(k-1)$. Thus, for $n>50(k-1)$ we have $|\ff|<|\mathcal C_3(n,k)|$.

Consider the case $n=C(k-1)\le 50(k-1)$. In this case, we need another lower bound on the size of $\mathcal C_3(n,k)$. The sum of all subtracted binomial coefficients in \eqref{eqsizec} is at most ${n-k\choose k-1}$, and thus
$$|\mathcal C_3(n,k)|>{n-1\choose k-1}-{n-k-1\choose k-1}-{n-k\choose k-1}\ge {n-1\choose k-1}-2{n-k\choose k-1}.$$
We have
$$\frac{{n-k\choose k-1}}{{n-1\choose k-1}}\le \Big(\frac{n-k}{n-1}\Big)^{k-1}\le e^{-\frac{(k-1)^2}{n-1}}\le  e^{-(k-1)/C}.$$
Thus, \begin{equation}\label{eqboundc2}|\mathcal C_3(n,k)|\ge \big(1-2e^{-(k-1)/C}\big){n-1\choose k-1}.\end{equation}
Let us upper bound $|\ff|$. We first deal with the case  $2k<n\le 7k$. If $n>2k$ then we have
$$\frac{{n-5\choose k-4}}{{n-5\choose k-1}}=\frac{(k-1)(k-2)(k-3)}{(n-k-1)(n-k-2)(n-k-3)}\le \frac{k-2}{n-k-1}\le 1-\frac 2k$$
and if additionally $k\ge 100$, then we have
$$\frac{{n-1\choose k-1}}{{n-5\choose k-1}}= \prod_{i=1}^4\frac{n-i}{n-k-i+1}\le 2^4.$$
Using the first inequality of \eqref{eqboundf} and these calculations, we  have $$|\ff|\le {n-1\choose k-1}-\frac 2k{n-5\choose k-1}\le \Big(1-\frac 1{8k}\Big){n-1\choose k-1}.$$
Assume that $2k< n\le 7k$. Comparing the bound \eqref{eqboundc2} and the upper bound on $|\ff|$, we see that $\frac 1{8k}>2e^{-(k-1)/7}$ for any $k\ge 100$, and thus $|\ff|<|\mathcal C_3(n,k)|$ in this case.

We are left to deal with the case $7k<n\le 50(k-1).$ For $n>7k$ we have ${n-5\choose k-1}-{n-6\choose k-1}={n-6\choose k-2}>{n-5\choose k-4}$ and
$$\frac{{n-1\choose k-1}}{{n-6\choose k-1}}= \prod_{i=1}^5\frac{n-i}{n-k-i+1}\le e^{5(k-1)/(n-k-4)}<e^{6k/n}.$$
Substituting this into the inequality \eqref{eqboundf}, we get $$|\ff|\le {n-1\choose k-1}-{n-6\choose k-1}\le \Big(1-e^{-6k/n}\Big){n-1\choose k-1}.$$
Comparing with \eqref{eqboundf}, we see that $e^{6k/n}<e^{7/C}<e^{(k-1)/C}/2$ for any $k\ge 100$ and $C\le 50$, and thus  $|\ff|<|\mathcal C_3(n,k)|$ again. The proof of the lemma is complete.
\end{proof}
In the remainder of this subsection, we assume that $n>2(k-1)^2$.
The expression \eqref{eqsizec} can be rewritten as follows
\begin{footnotesize}
\begin{align}\label{sizec2}
 \notag |\mathcal C_3(n,k)|\ =\ 3\ +\ &{n-3\choose k-3}+\ldots+{n-k-2\choose k-3} + \\
 \notag  & {n-4\choose k-3}+\ldots+{n-k-2\choose k-3}+\\
\notag   & {n-5\choose k-3}+\ldots+{n-k-3\choose k-3}+\\
 \notag  &\cdots \\
   &{n-k-2\choose k-3}+\ldots+{n-2k\choose k-3}\\
   \ge\ 3\ +\ &(k+(k-1)^2){n-k-2\choose k-3},
\end{align}
\end{footnotesize}
where in the inequality we used the convexity of ${x\choose \ell}$ as a function of $x$ 
and that we sum up ${x_i\choose k-3}$ with the average of $x_i$ being at least $n-k-2$. We also have
$$\frac{{n-3\choose k-3}}{{n-k-2\choose k-3}}=\prod_{i=1}^{k-3}\frac{n-2-i}{n-k-1-i}\le e^{(k-1)(k-3)/(n-2k+2)}\le e^{0.5}$$
for $n\ge 2(k-1)^2$. Thus, for such $n$, we have
\begin{equation}\label{eqc3large}|\mathcal C_3(n,k)|\ge e^{-0.5}(k^2-k+1){n-3\choose k-3}.\end{equation}

Our next goal is to upper bound $|\ff|$, using $\tau(\ff)\ge 3$ and $\gamma(\ff)> {n-5\choose k-3}$. To this end, we employ the peeling procedure, developed in \cite{KuZa} and \cite{Kup77}. We need some preparations.

We say that an intersecting family $\G$ is {\it maximal} if whenever $A\subsetneq B\in \G$ then $\G\setminus \{B\}\cup \{A\}$ is not intersecting. Moreover, we require a maximal $\G$ to be an antichain: we have $B_1\not\subset B_2$ for any $B_1,B_2\in \G$.
\begin{obs} Given an intersecting family $\ff$, there is a maximal intersecting family $\G$ such that for any $F\in \ff$ there is $G\in \G$ with $G\subset F$.
\end{obs}
The proof is straightforward: gradually replace sets from $\G$ by their proper subsets as long as the intersecting property is preserved. The family $\G$ is similar to the concept of a generating set of Ahlswede and Khachatrian, as well as to earlier concepts of bases studied by Frankl and F\"uredi. We refer to our recent survey \cite[Section 7]{Kup99} for a discussion of different concepts of bases.

For a real number $r\ge 1$ we say that a family $\ff$ is {\it $r$-spread} if $|\ff(X)|\le r^{-|X|}|\ff|$ for any set $X$. The following lemma is standard (see, e.g., \cite{KuZa}, \cite{Kup77}).
\begin{lem}\label{lemspread}
  Given $r\ge 1$ and $\G\subset{[n]\choose k}$, if $|\G|> r^k$ then there is a set $X$ of size strictly smaller than $k$ such that $\G(X)$ is $r$-spread.
\end{lem}
To prove the lemma, take an inclusion-maximal $X$ that violates the $r$-spreadness of $\G$.
\begin{lem}\label{lemspread2}
   If $\G\subset{[n]\choose \le m}$ is intersecting and there is a set $X$ of size strictly smaller than $m$ and a subfamily $\G'\subset \G$ such that $\G'(X)$ is $\alpha$-spread with $\alpha>m$, then $\G\setminus \G[X] \cup \{X\}$ is intersecting.
\end{lem}
\begin{proof}
Arguing indirectly, assume that a set $F\in \G$ is disjoint from $X$. Then $F$ must intersect every set from $\G'(X)$. However, using $\alpha$-spreadness of $\G'(X)$, we see $\sum_{y\in F}|\G'(X\cup \{y\})|\le \alpha^{-1}|F||\G'(X)|<|\G'(X)|$, a contradiction.
\end{proof}
The peeling procedure is as follows. We put $\mathcal T_k = \ff$ and then for each $i=k,k-1,\ldots, 2$ do the following.
\begin{itemize}
  \item Replace $\mathcal T_i$ by a maximal intersecting family $\mathcal T'_i$.
  \item Put $\mathcal W_i:=\mathcal T'_i\cap {[n]\choose i}$ and $\mathcal T_{i-1}:=\mathcal T'_i\setminus \mathcal W_i$.
\end{itemize}
We note the following properties of this peeling. First, $\mathcal T_i\subset {[n]\choose \le i}$. Second, for any $i$ we have
$$\ff=\ff[\mathcal T_i]\cup \bigcup_{j=i+1}^k\ff[\mathcal W_j].$$
Third, by Lemma~\ref{lemspread2}, there is no $X$ of size $<i$, such that  $\mathcal W_i(X)$ is $r$-spread with $r>i$. Lemma~\ref{lemspread} implies
$$|\mathcal W_i|\le i^i,$$
and thus
$$|\ff(\mathcal W_i)|\le i^i{n-i\choose k-i}=:g(i).$$
Let us compare $g(i)$ and $g(i-1)$ for $i\le k$.
We have
$$\frac{g(i)}{g(i-1)}\le \frac{i^i}{(i-1)^{i-1}}\cdot\frac{k-i}{n-i}\le ei \frac{k-i}{n-i}<\frac{ek^2}{4(n-k)}<\frac 12,$$
since $k\ge 100$ and $n>2(k-1)^2$.
Therefore, we may conclude that, for any $i\ge 1$, we have $$\sum_{j=i}^k g(j)\le 2g(i).$$
Using $n>2(k-1)^2$, we have
$$2g(5) = 2\cdot 5^5{n-5\choose k-5}\le 2\cdot 5^5 \frac{(k-3)(k-4)}{(n-k-1)(n-k-2)}{n-5\choose k-3}\le\frac{2\cdot 5^5}{4(k-1)^2}{n-5\choose k-3}<{n-5\choose k-3}.$$
This implies that $|\cup_{i=5}^k \ff[\mathcal W_i]|<{n-5\choose k-3}$, and thus $\mathcal T_4$ cannot consist of a singleton. Indeed, if this is the case, then  $\gamma(\ff)\le |\cup_{i=5}^k \ff(\mathcal W_i)|<{n-5\choose k-3}$, a contradiction with $\gamma(\ff)>{n-5\choose k-3}$.

Next, we analyze $\mathcal T_4=\mathcal T'_5\cap {[n]\choose \le 4}$. Since $\mathcal T'_5$ is maximal intersecting, there is no $>5$-spread subfamily $\mathcal T_4(X)$. Let us analyze the size of different layers of $\mathcal T_4(X)$, using the previous observation concerning spread subfamilies and that $\mathcal T_4$ is intersecting. It has no singletons, otherwise the family $\mathcal T_4(X)$ consists of that singleton only by the intersection property. The subfamily of $2$-element sets is intersecting and thus can be either a triangle or a star with $\ell$ petals. If $\ell>4$ then each $\le 4$-element set intersecting the star must intersect its center, and we get that all sets in the family $\mathcal T_4$ must contain the center of the star. Therefore, $\mathcal T_4$ can contain at most $4$ sets $B_1,\ldots, B_\ell$ of size $2$ (with $\ell\le 4$). Since $\mathcal T_4$ has no $>5$-spread subfamily, Lemma~\ref{lemspread} implies that it contains at most  $3^5$ sets of size $3$ and $4^5$ sets of size $4$.

Combining all these and using that $2f(5)<{n-5\choose k-3}<{n-3\choose k-3}$ and $n>2(k-1)^2\ge 198(k-1)$, we get that
\begin{align*}|\ff|\le & \sum_{i=1}^\ell|\ff[B_i]|+3^5{n-3\choose k-3}+4^5{n-4\choose k-4}+2f(5)\\
\le & \sum_{i=1}^\ell|\ff[B_i]|+\Big(3^5+\frac{4^5(k-4)}{n-4}+1\Big){n-3\choose k-3}\\
\le & \sum_{i=1}^\ell|\ff[B_i]|+250{n-3\choose k-3}.
\end{align*}
In order to bound $\ff[B_i]$, we use $\tau(\ff)\ge 3$. Namely, for each $i\in[\ell]$ there is a set $F_i$ such that $F_i\cap B_i=\emptyset$, and thus $\ff[B_i]= \cup_{x\in F_i}\ff[B_i\cup \{x\}]$. We have $|\ff[B_i\cup\{x\}]|\le {n-3\choose k-3}$, and thus
$\sum_{i=1}^\ell|\ff[B_i]|\le 4k{n-3\choose k-3}$.
Overall, we get that
$$|\ff|\le (4k+250){n-3\choose k-3}.$$
Comparing this with \eqref{eqc3large}, we see that, for $k\ge 100$,
$e^{-0.5}(k^2-k+1)>50k>4k+250$, and thus $|\ff|<|\mathcal C_3(n,k)|$. This completes the proof in the case $\gamma(\ff)>{n-5\choose k-3}$.
\subsection{The case $\gamma(\ff)\le {n-5\choose k-3}$}
We note that this part of the argument works for any $n>2k\ge 8$.  W.l.o.g. assume that $1$ has the largest degree in $\ff$.  The proof is based on the bipartite switching idea (see the sketch of the proof for more details).  We shall transform our family $\ff$ into another family (denoted by $\ff''$), which will satisfy $\tau(\ff'')=3$, $|\ff''|\ge |\ff|$ (with strict equality in case $\ff''$ is not isomorphic to $\ff$). Moreover, $\ff''(\bar 1)$ will have covering number $2$ and will be minimal with respect to that property.

To that end, take any $\mathcal M=\{M_1,\ldots, M_z\}\subset \ff(\bar 1)$ such that $\tau (\mathcal M)=2$ and $\mm$ is minimal w.r.t. this property. Remark that $z\ge 3$ due to the fact that $\tau(\mm)=2$ and $\mm$ is intersecting. Since $\mm$ is minimal, for each $M_\ell\in \mm$, there is \begin{equation}\label{eqil} i_\ell\in \Big(\bigcap_{M\in \mm\setminus \{M_\ell\}}M\Big),\end{equation}
where $\bigcap_{M\in \mm}M=\emptyset$ since $\tau(\mm)=2$. Fix an arbitrary choice of $i_\ell$ and put $I = \{i_\ell: \ell\in [z]\}$. W.l.o.g. assume that $I = [2,z+1]$. 
For each $i\ge 2$, consider the following bipartite graph $G_i$ (with the convention that $[2,1]=\emptyset$). The parts of $G_i$ are
\begin{align*}
\mathcal P_a^{i}:=\ &\Big\{P\ :\  P\in {[2,n]\choose k-1},\ P\cap [2,i]=\{i\}\Big\},\\
\mathcal P_b^{i}:=\ &\Big\{P\ :\  P\in {[2,n]\choose k},\ P\cap [2,i]=[2,i-1]\Big\},
\end{align*}
and edges connect disjoint sets. We identify $\mathcal P_a^i$ with ${[i+1,n]\choose k-2}$ and $\mathcal P_b^i$ with ${[i+1,n]\choose k-i+2}$.

We have $|\mathcal P_b^i\cap \ff(\bar 1)|\le |\ff(\bar 1)|\le {n-5\choose k-3}$. Thus, for each $i=2,\ldots, z+1$ we can apply \eqref{eqcreasy2} to
\begin{align*}
\aaa:=\ff(1)\cap \mathcal P_a^i \ \ \ \ \text{and}\ \ \ \ \
\bb:=\ff(\bar 1)\cap \mathcal P_b^i
\end{align*}
with $a=k-2$, $b=j=k-i+2$, and $n=n-i$. Note that $\bb$ already contains one set $M\in \mm$, and  thus we get $$|\aaa|+|\bb|\le {n-i\choose k-2}-{n-k-2\choose k-2} +1,$$
with a strict inequality unless $|\bb|=1$. We replace $\aaa,\bb$ with $\big\{F\cup \{i\}: F\in {[i+1,n]\choose k-2}, F\cap M\ne \emptyset\big\}$ and $\{M\}$, respectively, getting a new family $\ff_i$. Note that $|\ff_i|\ge |\ff_{i-1}|$, where $\ff_1 :=\ff$ and any such inequality is strict unless the two families on the two sides of the inequality coincide. Moreover, note that $\ff_i$ stays intersecting, since $\ff_i(\bar 1)$ consists of $\mathcal M$ and sets that contain $[2,i]$ entirely. All these sets intersect all the sets newly added to $\ff_i(1)$. We repeat the same exchange for any choice of set of representatives $I$.
 At the end, we get a family $\ff'$ with $\ff'(\bar 1)$ consisting of $\mm$ and some family $\mathcal U\subset \ff(\bar 1)$ of sets that all contain the set $I'$ of all elements that belong to all but $1$ set in $\mathcal M$. Indeed, should it contain another set, say $X$, which does not contain an element $i$, then the corresponding family $\mathcal B:=\ff(\bar 1)\cap \mathcal P_b^i$ would contain at least $2$ sets: $X$ and one of the sets of $\mathcal M$. This contradicts the fact that we performed the exchange for $i$. Let us w.l.o.g. assume that  $I' = [2,t]$.  By the above, we also get $|\ff'|>|\ff|$ unless $\ff'$ is isomorphic to $\ff$.


If $\mm$ is isomorphic to $\mathcal T_2(k)$ (cf. \eqref{deft2}), then the number of elements contained in exactly two sets (all but one sets) is $k+1$, and thus we may immediately conclude that $\mathcal U=\emptyset$: no $k$-set can contain a subset of size $k+1$. Otherwise, $\mm$ is not isomorphic to $\mathcal T_2(k)$.\footnote{This is only needed for the uniqueness of the extremal family $\mathcal C_3(n,k)$, since some of the exchanges we shall perform below may not necessarily strictly increase the size. But this does not pose problems since we will eventually arrive at a family $\ff''$ with $\ff''(\bar 1)=\mm$, which we will show to have size strictly smaller than that of $\mathcal C_3(n,k)$.}

Let us show that we may continue the transformations and assure that $\mathcal U$ is empty. It is clear if $t\ge k+2$: again, no $k$-element set can contain a $(k+1)$-element set as a subset. Otherwise, consider the family $\mm':=\{M\setminus [2,t]\ :\  M\in\mm\}$ and note that sets in $\mm'$ have size at least $1$. If there is no element $i'\in [t+1,n]$ that is contained in at least $2$ sets of $\mm'$, then take two elements $i\in M'$ and $j\in M''$, where $t+1\le i<j\le n$ and $M',M''$ are distinct sets in $\mathcal M'$, and perform the $(i,j)$-shift on $\ff'$. Only two sets in $S_{ij}(\mathcal M)$ will contain $i$, and thus  $\tau(S_{ij}(\mathcal M))=2$.  Moreover,
$S_{ij}(\ff')$ is intersecting due to the properties of shifting. Thus, we may replace $\ff'$ with $S_{ij}(\ff')$ and $\mm$ with $S_{ij}(\mm)$.

Now, we assume that there is an element in $i'\in [t+1,n]$ that is contained in at least $2$ sets of $\mm'$. 
Take a hitting set\footnote{That is, $I$ such that $I\cap M\ne \emptyset$ for any $M\in \mm'$.}  $I$ for $\mathcal M'$ of size at most $z-1$ and that contains $i'$. Note that such $I$ exists since $|\mathcal M'(\bar{i'})|\le z-2$.
 \begin{obs}\label{obslast} Consider a set $X$ such that $i'\notin X$ and $M\setminus X\ne \emptyset$ for any $M\in  \mm'$. Then there is such hitting set $I$ for $\mathcal M$ that is additionally disjoint from $X$.
  \end{obs} Indeed,  we may form $I$ by including $i'$ and one element from each of $M\setminus X$ for $M$'s that do not contain $i'$. 

Consider the bipartite graph $G(t,I)$ with parts
\begin{align*}
\mathcal P_a(t,I):=\ &\Big\{P \ :\ P\in {[2,n]\choose k-1},\ I\subset P,\  [2,t]\cap P=\emptyset\Big\},\\
\mathcal P_b(t,I):=\ &\Big\{P\ :\ P\in {[2,n]\choose k},\ [2,t]\subset P,\ I\cap P=\emptyset\Big\},
\end{align*}
and edges connecting disjoint sets. Put $Y = [t+1,n]\setminus I$. We identify $\mathcal P_a(t,I)$ with ${Y\choose k-|I|-1}$ and $\mathcal P_b(t,I)$ with ${Y\choose k-t+1}$. We have $t-1\ge z\ge |I|+1$, and, therefore, we may apply  \eqref{eqcreasy} to
\begin{align*}
\aaa:=\ff(1)\cap \mathcal P_a(t,I) \ \ \ \ \text{and}\ \ \ \ \
\bb:=\ff(\bar 1)\cap \mathcal P_b(t,I)
\end{align*} with $a:=k-|I|-1,\ b:=k-t+1$  and conclude that $|\aaa|+|\bb|\le {|Y|\choose k-|I|-1}$.
Replacing $\aaa$ with $\mathcal P_a(t,I)$ and $\bb$ with $\emptyset$ does not decrease the sum of sizes of the families and preserves the intersecting property of the family. (Here we also note that, by the choice of $I$, we have $|\mm\cap \mathcal P_b(t,I)| =\emptyset$.)

 Recall that $\mathcal U = \{F\in \ff(\bar 1): [2,t]\subset F\}$. We perform the same exchange operations for all possible choices of $I$.  We conclude that $\mathcal U$ does not contain sets that avoid $I$, for any allowed transversal $I$. The sets $F$ in $\mathcal U$ thus fall into two categories. First, $F$ may contain $i'$. Second, if $F$ does not contain $i'$, it must contain some $M\in \mm'$  by Observation~\ref{obslast}. (Otherwise, we find a hitting set $I$ that avoids $F$, and thus $F\in \bb$ in the above terms.) The latter is, however, impossible, since it would again imply that a $k$-element set from $\mathcal U$ contains a $(\ge k+1)$-element set $M\cup [2,t]$.

W.l.o.g., assume that $i' = t+1$. Therefore, we may assume that all sets in $\mathcal U$ contain $i'=t+1$, and thus all contain $[2,t+1]$. Next, we may perform similar exchange operations. Let us prepare the setup first. Slightly abusing notation, consider the family $\mm':=\{M\setminus [2,t+1]\ :\  M\in\mm\}$ and consider all possible transversals $I$ for $\mm'$ of size at most $z$. 
Consider the same bipartite graph with parts $\mathcal P_a(t+1, I)$ and $\mathcal P_b(t+1,I)$ and do the same exchange operations. The only condition we needed to obey is that on uniformity, which is $t\ge z+1\ge |I|+1$ in this case. (The reason it works now is the extra fixed element in $\mathcal P_b'(t+1,I)$, which makes the number of fixed elements in $\mathcal P_b(t+1,I)$  at least as big as that in $\mathcal P_a(t+1,I)$.)

Repeating this for all possible choices of $I$, we arrive at the family $\ff''$ and a situation where any set from $\ff''(\bar 1)\setminus\mm$ must intersect {\it any} such set $I$. By an analogue of Observation~\ref{obslast}, this is only possible for a set $F$ if $F\supset M\cup [2,t+1]$ for $M\in \mm'$. But this implies that $|F|>k$, which is impossible. Thus $\ff''(\bar 1) = \mm$, and so $\mathcal U$ is empty.
Moreover, $|\ff''|\ge |\ff|$ and it is not difficult to check $\tau (\ff'')=3$. 

Finally, we need to show that, among all {\it minimal} families $\mathcal M$, the choice of $\mathcal T_2(k)$ is the unique optimal. But this is a direct application of the second part of Lemma~\ref{lemmin} with $s=k$ and $[2,n]$ playing the role of $[m]$. Note that $n>2k$, and $m\ge 2k$. The proof of Theorem~\ref{thmtau3} in case $\gamma(\ff)\le {n-5\choose k-3}$ is complete.\\

{\sc Acknowledgements:} I would like to thank Peter Frankl for introducing me to the area and for numerous interesting discussions we had on the topic. The research was supported by the Ministry of Economic Development of the Russian Federation (agreement with MIPT No. 139-15-2025-013, dated June 20, 2025, IGK 000000C313925P4B0002).


\begin{thebibliography}{111}
\bibitem{AR} A. Arman and T. Retter, {\it An upper bound for the size of a k-uniform
intersecting family with covering number k}, J. Comb. Theory Ser. A
147 (2017), 18--26.

\bibitem{Bol} B. Bollob\'as, {\it On generalized graphs}, Acta Math. Acad. Sci. Hungar 16 (1965), 447--452.


\bibitem{Che} D.~D. Cherkashin, {\it On hypergraph cliques with chromatic number 3}, Moscow J. Comb. Numb. Th. 1 (2011), N3.

\bibitem{DT} S. Das and T. Tran, {\it Removal and Stability for Erd\H os--Ko--Rado}, SIAM J. Disc. Math. 30 (2016), 1102–-1114.


\bibitem{EKL} D. Ellis, N. Keller and N. Lifshitz, {\it Stability versions of Erd\H os--Ko--Rado type theorems via isoperimetry}, Journal of the European Mathematical Society 21 (2019), N12, 3857--3902.


\bibitem{EKR} P. Erd\H os, C. Ko, R. Rado, \textit{Intersection theorems for systems of finite sets}, The Quarterly Journal of Mathematics, 12 (1961), N1, 313--320.

\bibitem{EL} P. Erd\H os, L. Lov\'asz, {\it Problems and results on 3-chromatic hypergraphs
and some related questions}, in: Infnite and Finite Sets, Proc. Colloq.
Math. Soc. J\'anos Bolyai, Keszthely, Hungary (1973), North-Holland,
Amsterdam (1974), 609--627.

\bibitem{F16} P. Frankl, {\it On intersecting families of finite sets}, J. Comb. Theory Ser. A 24 (1978), N2, 146--161.

\bibitem{F3} P. Frankl, \textit{The shifting technique in extremal set theory}, Surveys in combinatorics, Lond. Math. Soc. Lecture Note Ser. 123 (1987), 81--110, Cambridge University
Press, Cambridge.

\bibitem{Fra1} P. Frankl,  \textit{Erd\H os--Ko--Rado theorem with conditions on the maximal degree}, J. Comb. Theory Ser. A 46 (1987), N2, 252--263.



\bibitem{Fra21} P. Frankl, {\it A near exponential improvement on a bound of Erd\H os and Lov\' asz}, Comb. Probab. Comput. 28 (2019), N5, 733--739.

\bibitem{Fdiv} P. Frankl, {\it Maximum degree and diversity in intersecting hypergraphs}, Journal of Combinatorial Theory, Series B 144 (2020), 81--94.

\bibitem{FK1} P. Frankl, A. Kupavskii, {\it Erd\H os--Ko--Rado theorem for $\{0, \pm 1\}$-vectors}, J. Comb. Theory Ser. A 155 (2018), 157--179.

\bibitem{FOT1} P. Frankl, K. Ota and N. Tokushige, {\it Uniform intersecting families with covering number four}, J. Comb. Theory Ser. A 71 (1995), 127--145.
\bibitem{FOT2} P. Frankl, K. Ota and N. Tokushige, {\it Covers in uniform intersecting families and a counterexample
to a conjecture of Lov\'asz}, J. Comb. Theory Ser. A 74 (1996), 33--42.

\bibitem{FT7} P. Frankl and N. Tokushige, {\it Invitation to intersection problems for finite sets}, J. Comb. Theory Ser. A 144 (2016), 157--211.

\bibitem{FWan} P. Frankl, J. Wang, {\it Improved bounds on the maximum diversity of intersecting families}, European Journal of Combinatorics 118 (2024).

\bibitem{FW23} P. Frankl, J. Wang, {\it Intersecting families with covering number three}, J. Comb. Theory Ser. B 171 (2025), 96--139.

\bibitem{Fri} E. Friedgut, {\it On the measure of intersecting families, uniqueness and stability}, Combinatorica, 28 (2008), 503--528.


\bibitem{Gold} J.L. Goldwasser, {\it Erd\H os--Ko--Rado with conditions on the minimum complementary degree}, J. Comb. Theory Ser. A 109 (2005), 45--62.


\bibitem{HK} J. Han, Y. Kohayakawa, {\it The maximum size of a non-trivial intersecting uniform family that is not a subfamily of the Hilton--Milner family}, Proc. Amer. Math. Soc. 145 (2017) N1, 73--87.

\bibitem{HM} A.J.W. Hilton, E.C. Milner, \textit{Some intersection theorems for systems of finite sets}, Quart. J. Math. Oxford 18 (1967), 369--384.


\bibitem{KostM}A. Kostochka, Dhruv Mubayi, {\it The structure of large intersecting families} Proc. Amer. Math. Soc.  145 (2017), N6,  2311--2321.

\bibitem{Kup21} A. Kupavskii, {\it Diversity of intersecting families}, European Journal of Combinatorics 74 (2018), 39--47.

\bibitem{Kup73} A. Kupavskii, {\it Structure and properties of large intersecting families}, unpublished, arXiv:1810.00920

\bibitem{Kup74} A. Kupavskii, {\it Structure of non-trivial intersecting families}, to appear in Proc. Amer. Math. Soc., arXiv:2412.07974

\bibitem{Kup99} A. Kupavskii, {\it Delta-system method: a survey}, arXiv:2508.20132


\bibitem{Kup77} A. Kupavskii, {\it An almost complete $t$-intersection theorem for permutations}, preprint

\bibitem{KZ} A. Kupavskii, D. Zakharov, {\it Regular bipartite graphs and intersecting families}, J. Comb. Theory Ser. A 155 (2018), 180--189.
\bibitem{KuZa} A. Kupavskii and D. Zakharov, {\it Spread approximations for forbidden intersections problems},  to appear in Advances in Mathematics, available at arxiv:2203.13379

\bibitem{Z} D. Zakharov, {\it On the size of maximal intersecting families}, Combinatorics, Probability and Computing 33 (2024), N1, 32--49.

\end{thebibliography}
\end{document}